\documentclass[reqno]{amsart}

\usepackage{amsfonts,amsmath,amstext,amssymb,latexsym,color}
\usepackage[utf8]{inputenc}

\newtheorem{theorem}{Theorem}[section]
\newtheorem{proposition}[theorem]{Proposition}

\newtheorem{lemma}[theorem]{Lemma}

\newtheorem{remark}[theorem]{Remark}
\newcommand{\dd}{\mathrm{d}}

\begin{document}
\title[Multiplicity of solutions for a scalar field equation]{Multiplicity of solutions for a scalar field equation
involving a fractional $p$-Laplacian with general nonlinearity}

\author[H. P. Bueno]{H. P. Bueno}\thanks{First author takes part in the project 422806/2018-8 by CNPq/Brazil}
\author[O. H. Miyagaki]{O. H. Miyagaki}\thanks{Second author was supported by Grant 2019/24901-3 by S\~ao Paulo Research Foundation (FAPESP) and Grant 307061/2018-3 by CNPq/Brazil.}
\author[A. L. Vieira]{A. L. Vieira}
\address[H. P. Bueno]{Departmento de Matemática, Universidade Federal de Minas Gerais, 31270-901 - Belo Horizonte-MG, Brazil}
\email{hamilton@mat.ufmg.br}
\address[O. H. Miyagaki]{Departmento de Matemática, Universidade Federal de S\~ao Carlos, 13565-905 - S\~ao Carlos-SP, Brazil}
\email{olimpio@ufscar.br, ohmiyagaki@gmail.com}
\address[A. L. Vieira]{Departmento de Matemática, Universidade Federal de Minas Gerais, 31270-901 - Belo Horizonte-MG, Brazil and Departmento de Ciências Exatas, UFVJM, 39803-371 - Te\'ofilo Otoni, Brazil}
\email{vieira.ailton@ufvjm.edu.br} 
\subjclass[2020]{35A15, 35R11,
35J62, 35B33} 
\keywords{Fractional $p$-Laplacian; symmetric
criticality; Moser-Trudinger inequality; exponential and polynomial
growth.}

\begin{abstract} We investigate the existence of infinitely many radially symmetric solutions to the following problem
$$(-\Delta_p)^s u=g(u) \ \ \textrm{ in } \ \ \mathbb{R}^N,  \ \ u\in W^{s,p}(\mathbb{R}^N),$$ where
$s\in (0,1)$, $2 \leq p < \infty$, $sp \leq N $, $2 \leq N \in
\mathbb{N}$ and $(-\Delta_p)^s$ is the fractional $p$-Laplacian
operator. We treat both of cases $sp=N$ and $sp<N.$ The nonlinearity
$g$ is a function of Berestycki-Lions type with critical exponential growth if $sp=N$ and critical polynomial growth if $sp<N$. We also prove the existence of a ground state solution for the same problem. 
\end{abstract}
\maketitle
\section{Introduction}
In this paper, we establish the existence of infinitely many
radially symmetric solutions to the problem
\begin{equation}\label{Prob1}
(-\Delta_p)^s u=g(u) \ \ \textrm{ in } \ \ \mathbb{R}^N,
\end{equation}
where $s\in (0,1)$, $2 \leq p < \infty$, $sp \leq N $, $2\leq N \in
\mathbb{N}$, $(-\Delta_p)^s$ is the fractional $p$-Laplacian
operator and $g$ is an odd continuous function satisfying some
properties.

Great attention has been devoted to the study of elliptic equations involving the fractional $p$-Laplacian operator in recent years.
Mainly when $p = 2$, it appears in many models arising from concrete
applications in Biology (e.g., population dynamics), Physics (e.g.,
continuum mechanics, phase transition phenomena), Game Theory and
Financial Mathematics, see \cite{VA,BBMP}. But the operator also
draws attention from a purely mathematical point of view, because of
the challenging difficulties due to its both nonlocal and nonlinear
character \cite{GP}.

The problem \eqref{Prob1} has already evolved into an elaborate
theory whose literature is too broad to attempt any comprehensive
synthesis on a single paper. We refer the interested reader to the
papers \cite{VA, BBMP, DPV, GP} and references therein.

By using variational methods, Hirata, Ikoma and Tanaka \cite{HIT} studied a particular limit case of problem \eqref{Prob1}. Namely, they studied multiplicity of solutions for $-\triangle u =g(u) \textrm{ in } H^1(\mathbb{R}^N), \quad N \geq 2$, slightly extending the results of Berestycki and Lions \cite{BL} in the case $N\geq 3$ and Berestycki, Gallou\"et and Kavian \cite{BGK} for $N=2$. Following a similar framework, Ambrosio establishes multiplicity of solutions for $(-\Delta)^su=g(u)$ in  $H^s(\mathbb{R}^N)$, if $N \geq 2$ and $s\in (0,1)$. The techniques
used in \cite{HIT,VA} depend heavily on the Poho\v{z}aev identity for the fractionary operator, which plays a crucial role for proving the compactness condition for the Palais-Smale sequences. Since a proof of the Poho\v{z}aev
identity for the fractionary $p$-Laplacian operator is still
unknown, we have looked for an alternative approach, adapting ideas of Zhang and Chen \cite{Zhang}. Furthermore, as in the papers \cite{VA,HIT}, we obtain our results without supposing the Ambrosetti-Rabinowitz condition.

Let us denote
$$p_{s}^*=\left\{\begin{array}{ll}\displaystyle\frac{pN}{N-sp}, &\textrm{if }\ sp<N;\vspace{.1cm}\\
    +\infty, &\textrm{ if } sp =N.
\end{array}\right.
$$
In the case $sp\leq N$, it is well known that $W_{r}^{s,p}(\mathbb{R}^N)$ is compactly embedded in $L^q(\mathbb{R}^N)$ for any $q\in (p,p_{s}^*)$ (see \cite[Theorem 1.1.11]{VA2} or \cite[Theorem II.1]{PLL}), the embedding being continuous if $q=p$ or $q=p_{s}^*$, if $p_{s}^*<\infty$.  We denote by $C_q>0$ the best constant of this Sobolev embedding, that is,
$$C_q\left[\int_{\mathbb{R}^N}{|u|^q}\right]^{\frac{p}{q}} \leq [u]_{s,p}^{p}+\int_{\mathbb{R}^N}{|u|^p}, \quad \forall u \in  W^{s,p}(\mathbb{R}^N).$$

Inspired by the papers Alves, Figueiredo and Siciliano \cite{AFS} and Alves, Souto and Montenegro \cite{ASM} - the first dealing with the fractionary Laplacian operator and the second considering the Laplacian operator - we write the nonlinearity $g$ in \eqref{Prob1} in the form $g(t)=-|t|^{p-2}t+f(t)$,  where $f$ is a continuous odd function satisfying
\begin{enumerate}
\item [$(f1)$] $\displaystyle\lim_{t \to 0}{\frac{f(t)}{|t|^{p-2}t}}=0$;
\item [$(f2)$]
$\displaystyle\limsup_{t \to
\infty}{\frac{f(t)}{|t|^{p_{s}^{*}-1}}}\leq 1, \quad sp<N$;
%\item [$(f3)$]
%$H(t)=f(t)t-pF(t) \geq 0, \textrm{ for every } t> 0$, where
%$F(t)=\int_{0}^{t}f(\tau)\dd \tau\geq 0$;
\item [$(f3)$] There exists $\mu >0$ and $q\in
(p,+\infty)$ such that $f(t)\geq \mu t^{q-1}, \ \forall\,t\geq 0$;
\item [$(f4)$] If $sp=N$, there exists $\alpha_0 >0$ such that
$$\lim_{t \to +\infty}{\frac{f(t)}{e^{\alpha
t^{\frac{N}{N-s}}}}}=0$$ (respectively, $=+\infty$), if $\alpha >
\alpha_0$ (respectively, $\alpha <  \alpha_0)$.
\end{enumerate}

The natural setting for problem \eqref{Prob1} is the fractional
Sobolev space
$$W^{s,p}(\mathbb{R}^{N}) := \{  u\in L^{p}(\mathbb{R}^N)\,:\, [u]_{s,p}^{p}< \infty \},$$
where
$$[u]_{s,p}^{p}:=\int_{}^{}\int_{\mathbb{R}^{2N}}^{}{\frac{|u(x)-u(y)|^{p}}{|x-y|^{N+sp}}}\dd x\dd y$$
is the Gagliardo seminorm of $u$. This space, endowed with the
natural norm
$\|u\|_{s,p}=([u]_{s,p}^{p}+\|u\|_{p}^{p})^{\frac{1}{p}}$ (where
$\|u\|_p=\|u\|_{L^p(\mathbb{R}^N)}$) is a reflexive Banach space,
see, e.g., \cite{DD}.

Due to Principle of Symmetric Criticality (see \cite{RSP}), it is sufficient to find
solutions to the problem \eqref{Prob1} in the closed subspace of
radial functions, that is, in the space
$$W_{r}^{s,p}(\mathbb{R}^{N}) := \{  u\in W^{s,p}(\mathbb{R}^N) ; u(|x|)=u(x)\}.$$

A weak solution of \eqref{Prob1} satisfies, for any $\phi \in
W_{r}^{s,p}(\mathbb{R}^N)$,
$$\langle (-\Delta_p)^s u, \phi\rangle +
\int_{\mathbb{R}^N}{|u|^{p-2}u\phi \dd x} =
\int_{\mathbb{R}^N}^{}{f(u)\phi \dd x},$$ where
$$\langle (-\Delta_p)^su, v\rangle= \iint_{\mathbb{R}^{2N}}^{}{\frac{|u(x)-u(y)|^{p-2}(u(x)-u(y))(v(x)-v(y))}{|x-y|^{N+sp}}\dd x\dd y}.$$

The ``energy" functional $I\in
\mathcal{C}^1\left(W^{s,p}_{r}(\mathbb{R}^N), \mathbb{R}\right)$
associated with problem \eqref{Prob1}  is defined by
$$I(u)=\frac{1}{p}\|u\|_{s,p}^p-\int_{\mathbb{R}^N}{F(u) \dd x}.$$
Since this functional has derivative
$$I'(u)\cdot v= \langle (-\Delta_p)^s u, v\rangle + \int_{\mathbb{R}^N}{|u|^{p-2}uv \dd x} - \int_{\mathbb{R}^N}^{}{f(u)v \dd x},$$ we see that critical points
of $I(u)$ are weak solutions to \eqref{Prob1}.

Our main result is the following
\begin{theorem}\label{teo1} Let  $s\in(0,1)$, $p\in [2,+\infty)$, $sp\leq N$ $(N\geq 2)$ and $g(t)=-|t|^{p-2}t+f(t)$. Then \eqref{Prob1} has infinitely many radially symmetric solutions $(u_n)_{n\in \mathbb{N}}$ such that $I(u_n) \to \infty$ as $n \to \infty$ if
    \begin{enumerate}
        \item [$(i)$] $sp=N$ and the function $f$ satisfies $(f1)$ and $(f3) - (f4)$;
        \item [$(ii)$] $sp<N$ and the function $f$ satisfies $(f1) - (f3)$.
    \end{enumerate}
    The same result is valid in the limit case $s=1$.
\end{theorem}

We also consider the existence of a ground state solution for problem \eqref{Prob1}. Since we are looking for positive solutions, we will assume $f(t)=0$ if $t\leq 0$.

We need an additional hypotheses and also to modify hypotheses ($f3$).
\begin{enumerate}
    \item [$(f3)_{gs}$] $f(t)\geq
    \mu t^{q-1}$, for all $t\geq 0$, where
    \begin{itemize}
        \item if $sp=N$, than $q >p$ and $\mu>\mu^*$, where $$\mu^*:=\left(\frac{q-p}{pq}\right)^{(q-p)/p}C^{q/p}_q.$$
        \item if $sp<N$, than $q\in (p, p_{s}^*)$ and $\mu>\mu^*$, where $$\mu^*=\left[ p^{(sp-N)/sp}S^{-N/sp}\frac{N}{s}
        \left(\frac{pN}{N-sp}\right)^{(N-sp)/sp} \right]^{(q-p)/p}\left(\frac{q-p}{pq}\right)^{(q-p)/p}C^{q/p}_q;$$
        \end{itemize}

    \item [$(f5)$] $H(t)=f(t)t-pF(t) \geq 0$, for every $t> 0$, where
    $F(t)=\int_{0}^{t}f(\tau)\dd \tau\geq 0$.
\end{enumerate}
\
\begin{theorem}\label{Thgs}Let $N\geq 2$. The problem \eqref{Prob1} admits a  non-negative, radially symmetric and decreasing ground state
    solution  if
    \begin{enumerate}
        \item $sp = N, \ N\geq 2$ and $f$ satisfies $(f1)$ and $ (f3)_{gs} - (f4)$;
        \item $sp < N, \ N\geq 2$ and $f$ satisfies $(f1)-(f2)-(f3)_{gs}-(f5)$.
    \end{enumerate}
\end{theorem}
The proof of Theorem \ref{Thgs} requires only minor changes in results obtained by Alves, Figueiredo and Siciliano \cite{AFS}. In the case $sp<N$, we need to apply the Poho\v{z}aev identity for the fractional $p$-Laplacian operator, a result which is not fully proved.
\section{Preliminaries}
We denote by $\|\cdot\|_q$ the usual norm in $L^q(\mathbb{R}^N)$.

For the reader's convenience, we recall the main results about the
Trudinger-Moser inequalities for the space
$W_{r}^{s,p}(\mathbb{R}^N)$. We start stating a important result due
to T. Ozawa \cite{TO}.
\begin{theorem}
    Let $s=\frac{N}{p}\in(0,1)$ and $1 < p < \infty$ satisfy
    $\frac{1}{s} + \frac{1}{p}= 1$. Then, there exist positive constants
    $\gamma$ and $C_\gamma$ such that, for all $u\in
    W_{r}^{s,p}(\mathbb{R}^N)$ satisfying $\|u\|_{s,p}\leq 1$, it holds
    \begin{equation}\label{intTO}
        \int_{\mathbb{R}^N}^{}{\left( e^{\gamma |u(x)|^{p'}} -
            \displaystyle\sum_{0\leq j <p,
                j\in\mathbb{N}}{\frac{1}{j!}{\left(\gamma|u(x)|^{p'}\right)^j}}
            \right)\dd x} \leq C_\gamma\|u\|_{p}^{p}
    \end{equation}
\end{theorem}

The next result is due to Adachi and Tanaka \cite{AT}.
\begin{theorem}\label{MTNp} If $\omega_{p-1}$ denotes the surface area of the unit sphere in $\mathbb{R}^p$
    and $\alpha_p=p\omega_{p-1}^{1/(p-1)}$, then, for any $N\geq 2$ and
    any $\alpha \in (0, \alpha_p)$,  there exists a constant
    $C_{\alpha}>0$ such that
    \begin{equation}\label{intAT}
        \int_{\mathbb{R}^p}^{}{\Phi_p\left( \alpha
            \left(\frac{|u(x)|}{\|\nabla u\|_p} \right)^{p'} \right)\dd x}\leq
        C_{\alpha} \frac{\|u\|_{p}^p}{\|\nabla u\|_{p}^p}, \ \  u\in
        W_{r}^{1,p}(\mathbb{R}^p)\backslash \{ 0 \},
    \end{equation}
    where
    $$\Phi_p(\eta)=e^{\eta}-\displaystyle\sum_{j=0}^{p-2}{\frac{1}{j!}\eta^j}.$$
\end{theorem}

\begin{remark}\label{RemarkTM} Note that, if $\|\nabla u\|_p\leq M$, the inequality \eqref{intAT} can be written into the form \eqref{intTO}.
\end{remark}

\begin{lemma}\label{limit}If $p'=\frac{p}{p-1}=\frac{N}{N-s}$, let us define
$$Q(t):=e^{\alpha |t|^{p'}}-T(t)\quad\text{and}\quad T(t)=\displaystyle\sum_{0\leq j <p, j\in \mathbb{N}}{\frac{1}{j!}{\left(\alpha |t|^{p'}\right)^j}}.$$
Then, in the case $sp=N$,
\begin{enumerate}
    \item [$(i)$] $(f4)$ implies $\displaystyle\lim_{ |t| \to +\infty}\frac{f(t)}{Q(t)}=0$;
    \item [$(ii)$] $(f3)$ and $(f4)$ imply $\displaystyle\lim_{ |t| \to +\infty}\frac{F(t)}{Q(t)}=0$.
\end{enumerate}
\end{lemma}
\begin{proof}Take $\alpha>\alpha_0$. Since
$$\lim_{ |t| \to +\infty}\frac{Q(t)}{e^{\alpha t^{p'}}}=\lim_{ |t| \to
        +\infty}\left(1-\frac{T(t)}{e^{\alpha t^{p'}}}\right)=1,$$
it follows from $(f4)$ that
$$\lim_{ |t| \to +\infty}\frac{f(t)}{Q(t)} =\lim_{ |t| \to +\infty}\frac{f(t)/{e^{\alpha t^{p'}}}}{Q(t)/e^{\alpha t^{p'}}}=0,$$
proving ($i$).

Moreover, by $(f3)$ there exists $\mu >0$ and $q>p$ such that
$f(t)\geq \mu t^{q-1}, \ \forall\ t\geq 0.$ It follows that
$F(t)\geq \displaystyle\frac{\mu t^q}{q}, \ \forall\,t\geq 0.$ Since
$F$ is an even function, we have $F(t)\geq \mu |t|^q/q$, for all
$t\in \mathbb{R}$. Thus, $F(t) \to +\infty$ as $|t| \to +\infty$.
Hence,
$$\displaystyle\lim_{ |t| \to +\infty}\frac{F(t)}{e^{\alpha |t|^{p'}}} = \displaystyle\lim_{ |t| \to +\infty}\frac{f(t)}{\frac{d}{dt} e^{\alpha |t|^{p'}}}=0,$$ which implies

$$\displaystyle\lim_{ |t| \to +\infty}\frac{F(t)}{Q(t)} = \displaystyle\lim_{ |t| \to +\infty}\frac{F(t)/{e^{\alpha t^{p'}}}}{Q(t)/e^{\alpha t^{p'}}}=0.$$

\vspace*{-.7cm}\end{proof}

The next result is an estimate for $|F(t)|$ adequate for our purposes. Its proof is an immediate consequence of Lemma \ref{limit}.
\begin{proposition}\label{DesgExp}
Suppose that $(f1)$, $(f3)$ and $(f4)$ are valid. Given $q>\frac{N}{s}=p$ for all $\alpha>\alpha_0,$ there exists $\varepsilon>0, \ C>0,$ such that
$$|F(t)|\leq \frac{\varepsilon|t|^p}{p}+ C Q(t)|t|^q, \quad \forall t \in \mathbb{R}.$$
\end{proposition}

\section{Proof of Theorem \ref{teo1}}
Our first result simply adapts the arguments of Theorem 10 in
Berestycki-Lions \cite{BL} for the space
$W_{r}^{s,p}(\mathbb{R}^N)$. Let $\mathbb{S}^{n-1}$ stand for the
unit sphere in $\mathbb{R}^N$.
\begin{theorem}
    For all $n\in \mathbb{N},$ there exists an odd continuous map
    $\pi_n \colon \mathbb{S}^{n-1} \to W_{r}^{s,p}(\mathbb{R}^N)$ such that
    \begin{enumerate}
        \item [$(i)$] $\pi_n(\sigma)$ is radially symmetric in $\mathbb{S}^{n-1}$;
        \item [$(ii)$] $0 \not\in \pi_n(\mathbb{S}^{n-1})$;
        \item [$(iii)$] $\displaystyle\int_{\mathbb{R}^N}^{}{F(\pi_n(\sigma)) \dd x} \geq 1$, for all  $\sigma \in \mathbb{S}^{n-1}$.
    \end{enumerate}
\end{theorem}

This result follows by considering the subset $V\subset
W_{r}^{s,p}(\mathbb{R}^N)$ and the functional $J\colon
W_{r}^{s,p}(\mathbb{R}^N) \to \mathbb{R}$  given by
$$V=\left\{v\in W_{r}^{s,p}(\mathbb{R}^N)\,:\, \displaystyle\int_{\mathbb{R}^N}{F(v)\dd x}=1\right\},\qquad
J(w)= \displaystyle\int_{\mathbb{R}^N}{F(w) \dd x}.$$ It follows
easily that $J$ is continuous and ($f3$) implies that
%$J(w)\geq \displaystyle\int_{\mathbb{R}^N}^{}\frac{\lambda|w|^q}{q}\dd x$, proving that
$J$ is coercive. Thus, there exists $w_1\in
W_{r}^{s,p}(\mathbb{R}^N)$ such that $J(w_1)>1$. Let $\bar B_{\ell}$
be the closed ball centered at the origin with radius
$\ell=\|w_1\|_{s,p}$. It follows from the Intermediate Value Theorem
the existence of $v\in \bar B_{\ell}$ such that $J(v)=1$, proving
that $V\neq\emptyset$.

Now the arguments of \cite[Theorem 10]{BL} are easily adapted to the
space
 $W_{r}^{s,p}(\mathbb{R}^N)$.

\begin{lemma}\label{lemgeo}The geometry of the Symmetric  Mountain
Pass Theorem is valid. In fact, we have
\begin{enumerate}
\item[$(i)$] There are $\beta, \ \rho>0$ such that
$I(u) \geq \beta
>0$ for   $\|u\|_{s,p}=\rho$ and $I(u) \geq 0$ for $\|u\|_{s,p}\leq \rho\;;$

\item[$(ii)$] If  $n\in\mathbb{N},$ there exists an odd continuous mapping $\gamma_n \colon \mathbb{S}^{n-1}\to W_{r}^{s,p}(\mathbb{R}^N)$ such that
$I(\gamma_n(\sigma))<0, \forall \sigma\in\mathbb{S}^{n-1}$.
\end{enumerate}
\end{lemma}

\begin{proof} Since $f$  is odd, it follows that $f(0)=0$ and $I(0)=0$.

Let us consider initially the case $sp=N$. Fixed any $\theta>p$, it
follows of Proposition \ref{DesgExp} that there exists
$\varepsilon>0$ small and $C>0$ such that
$$ |F(t)|\leq \frac{\varepsilon |t|^p}{p}+
CQ(t),\ \ \forall t\in \mathbb{R}.$$
Therefore,
\begin{align*}I(u)&\geq\frac{1}{p}\|u\|_{s,p}^{p}-\int_{\mathbb{R}^N}\left(\frac{\varepsilon}{p}|u|^p+C Q(|u|)|u|^{\theta}\right) \dd x =\frac{1}{p}\|u\|_{1}^{p}-C\int_{\mathbb{R}^N}{Q(|u|) |u|^{\theta} \dd x},
\end{align*}where
$$\|u\|_{1}^p=\frac{1}{p}[u]_{s,p}^{p}+\frac{1-\varepsilon}{p} \|u\|_{p}^p<C_1\|u\|_{s,p}$$ is a norm equivalent to
$\|\cdot\|_{s,p}.$

Hence, for any fixed $r>p$, we can take $\alpha_1\in(0,\alpha_0)$
such that $\alpha_1 r \in (0, \alpha_0)$. By applying H\"older's
inequality with $r'=r/(r-1)$, we obtain
$$I(u) \geq \frac{C_1}{p}\|u\|_{s,p}^{p}-C\left(\displaystyle\int_{\mathbb{R}^N} \left(Q(|u|)\right)^{r'} \dd x\right)^{\frac{1}{r'}}
  \left( \displaystyle\int_{\mathbb{R}^N}^{}{|u|^{\theta r}
  \dd x}\right)^{\frac{1}{r}}.$$

Suppose now that $\|u\|_{s,p} \leq 1$. It follows from the
Trudinger-Moser inequality (Theorem \ref{MTNp}, see also Remark \ref{RemarkTM}) that
$$I(u)  \geq  \frac{C_1}{p}\|u\|_{s,p}^{p}-CC_{\alpha r} \|u\|_{p}^{\theta} \geq  \|u\|_{s,p}^{p} \left(C'- \|u\|_{s,p}^{\theta-p}\right).
$$
Since $\theta-p>0$, we obtain ($i$) by taking $\|u\|_{s,p}$ small
enough.

Since $\pi_n$ is continuous, there exists $M>0$ such that
$\|\pi_n(\sigma)\|_{s,p}\leq M,$ for any $\sigma \in
\mathbb{S}^{n-1}$. Let us define, for $t\geq 1$,
$$\phi_{n}^{t}(\sigma)(x) = \pi_n(\sigma)\left(\frac{x}{t}\right):\mathbb{S}^{n-1}\to W_{r}^{s,p}(\mathbb{R}^N).$$
It follows from ($f3$) that
\begin{align*}I(\phi_{n}^{t})&=\frac{t^{N-sp}}{p}[\pi_n(\sigma)]_{s,p}^p+\displaystyle\frac{t^N}{p}\|\pi_n(\sigma)\|_{p}^p-t^N\displaystyle\int_{\mathbb{R}^N}{F(\pi_n(\sigma))\dd x}\\
&\leq\frac{1}{p}[\pi_n(\sigma)]_{s,p}^p+t^N\left(\frac{1}{p}\|\pi_n(\sigma)\|_{p}^p-\frac{\mu}{q} \|\pi_n(\sigma)\|_{q}^q\right)\\
&=\frac{1}{p}B_1^p+t^N\left(\frac{B_{2}^p}{p}-\frac{\mu
B_{3}^q}{q}\right).
\end{align*}
where $B_1=[\pi_n(\sigma)]_{s,p}$ $
B_2=\|\pi_n(\sigma)\|_{p}$ and $B_3=\|\pi_n(\sigma)\|_q$ are
constants. The term between parenthesis is negative if $\mu$ is
large enough. So, there is $\bar{t}>1$ such that
$I(\phi_{n}^{\bar{t}})< 0$, concluding the proof of ($ii$).

We now consider the case $sp<N$. It follows from our hypothesis
$(f1)$ and  $(f2)$ that there exists $\varepsilon>0$ small and
$C_{\varepsilon}>0$ such that
$$ |F(t)|\leq \frac{\varepsilon |t|^p}{p}+ C_{\varepsilon}|t|^{p_{s}^*},\ \ \forall t\in \mathbb{R}.$$ Thus, as
in the proof of Lemma \ref{lemgeo}, we obtain
\begin{align*}
%\[
I(u)&\geq\frac{1}{p}\|u\|_{s,p}^{p}-\int_{\mathbb{R}^N}{\displaystyle\frac{\varepsilon |u|^p}{p}+C_{\varepsilon}|u|^{p_{s}^*} \dd x} \\
&=\frac{1}{p}[u]_{s,p}^{p}+\displaystyle\frac{1-\varepsilon}{p}\|u\|_{p}^p-C_{\varepsilon}\int_{\mathbb{R}^N}{|u|^{p_{s}^*} \dd x} \\
&
\geq\frac{C}{p}\|u\|_{s,p}^{p}-\tilde{C}\|u\|_{s,p}^{p_{s}^*}.
\end{align*}
The rest of the proof is analogous.

\end{proof}

Adapting some ideas of Zhang and Chen \cite{Zhang} we
obtain the next result.
\begin{lemma}\label{lemLTDA}
 Any $(PS)_c$-sequence $(u_j)$ is bounded in $W_{r}^{s,p}(\mathbb{R}^N)$.
\end{lemma}
\begin{proof} Let $(u_j)$ be a $(PS)_c$ sequence. Then we have \begin{align*}\frac{1}{p}\|u_j\|_{s,p}^p-\int_{\mathbb{R}^N}{F(u_j) \dd x}&\to c,\\
\langle (-\Delta_p)^s u_j, \phi\rangle +
    \int_{\mathbb{R}^N}{|u_j|^{p-2}u_j\phi \dd x} -
    \int_{\mathbb{R}^N}{f(u_j)\phi \dd x}&=o(1)\|\phi\|,
\end{align*}
for any $\phi\in W_{r}^{s,p}(\mathbb{R}^N)$.

By contradiction, passing to a subsequence if necessary, we can suppose that $0<\|u_j\|_{s,p} \to +\infty$ as $j\to \infty$. We then set
$$v_j=\frac{u_j}{\|u_j\|_{s,p}}, \ j\in\mathbb{N}.$$
Hence, $v_j \rightharpoonup v \in W_{r}^{s,p}(\mathbb{R}^N)$. Therefore, for  all $q\in (p,p_{s}^*)$, we have $v_j \to v$ in $L^q(\mathbb{R}^N)$ and also $v_j(x)\to v(x)$ a.e. in $\mathbb{R}^N.$

Let us suppose initially that $v\neq 0$. Then, $\Theta = \{x \in \mathbb{R}^N\,:\, v(x) \neq 0\}$ has positive Lebesgue measure and $|u_j(x)|= |v_j(x)|\, \|u_j(x)\|_{s,p} \to \infty, \ \forall x \in \Theta$.

It follows from ($f3$) that $F(v_j)\geq 0$ for all $j$ and
$$\frac{F(u_j)}{\|u_j\|_{s,p}^p}\geq\frac{\mu |u_j|^q}{\|u_j\|_{s,p}^p}=\frac{\mu |v_j|^q \|u_j\|_{s,p}^q}{\|u_j\|_{s,p}^p}=\mu |v_j|^q \|u_j\|_{s,p}^{q-p}.$$

Hence, it follows from Fatou's Lemma that
$$\liminf_{j\to \infty}{\int_{v\neq 0}^{}{\frac{F(u_j)}{\|u_j\|_{s,p}^p}\dd x}} \geq \int_{v\neq 0}{\liminf_{j\to \infty}{\mu |v_j|^q \|u_j\|_{s,p}^{q-p}}\dd x}=\infty.$$ But we also have
$$(c+o(1)) =\frac{1}{p}\|u_j\|_{s,p}^p- \int_{\mathbb{R}^N}^{}{F(u_j)\dd x}$$ and
$$\frac{1}{p}-\frac{(c+o(1))}{\|u_j\|_{s,p}^p} = \int_{\mathbb{R}^N}^{}{\frac{F(u_j)}{\|u_j\|_{s,p}^p} \dd x} \geq \int_{v\neq 0}^{}{\frac{F(u_j)}{\|u_j\|_{s,p}^p}\dd x}.$$
Thus, we have reached a contradiction.

Suppose now that $v = 0$. Since $I'(u_j)\cdot \phi =
o(1)\|\phi\|_{s,p}$, dividing this expression by $\|u_j\|^{p-1}_{s,p}$ we obtain
\begin{equation}\label{eqp1}
    \langle (-\Delta_p)^sv_j, \phi \rangle +\int_{\mathbb{R}^N}{\frac{|u_j|^{p-2}u_j\phi}{\|u_j\|_{s,p}^{p-1}} \dd x}-\int_{\mathbb{R}^N}^{}{\frac{f(u_j)\phi}{\|u_j\|_{s,p}^{p-1}} \dd x}= \frac{o(1) \|\phi\|_{s,p}}{\|u_j\|_{s,p}^{p-1}}.
\end{equation}
Passing to the limit in \eqref{eqp1} as $j\to \infty$, we conclude that
$$\int_{\mathbb{R}^N}^{}{\frac{f(u_j)\phi}{\|u_j\|_{s,p}^{p-1}} \dd x}\to 0,
\ \forall \phi \in W_{r}^{s,p}(\mathbb{R}^N).$$ Thus, there exists a
constant $C>0$ such that
\begin{equation*} \left|\int_{\mathbb{R}^N}^{}{\frac{f(u_j)}{\|u_j\|_{s,p}^{p-1}}\phi \dd x}\right| \leq C\|\phi\|_{s,p}, \  \forall \phi \in W_{r}^{s,p}(\mathbb{R}^N).
\end{equation*}

For every $j \in \mathbb{N},$ set
$$ T_j(\phi)=\int_{\mathbb{R}^N}^{}{\frac{f(u_j)\phi}{\|u_j\|_{s,p}^{p-1}} \dd x}, \ \phi \in W_{r}^{s,p}(\mathbb{R}^N).$$

It follows that $\{T_j\}$ is a family of bounded linear functionals
and $\displaystyle\sup_{j\in \mathbb{N}}{|T_j(\phi)|} \leq C$ for
all $\phi \in W_{r}^{s,p}(\mathbb{R}^N)$. Hence,
$$\sup_{j\in \mathbb{N}}{\|T_j\|} < \infty.$$
Taking into account the embedding $W_{r}^{s,p}(\mathbb{R}^N)
\hookrightarrow L^q(\mathbb{R}^N)$ for $q\in [p,+\infty)$
(respectively, $q\in [p,p_{s}^*]$ if $sp<N$), the Hahn-Banach
Theorem guarantees the existence of a continuous linear functional
$\tilde{T}_j$ defined in $L^q(\mathbb{R}^N)$ such that
$\tilde{T}_j(\phi)= T_j(\phi)$ and
$\|\tilde{T}_j\|_{({L^{q}(\mathbb{R}^N)})^*} = \|T_j\|$, for all
$\phi \in W_{r}^{s,p}(\mathbb{R}^N)$.

Thus, there exist functions $h_j \in L^{q'}(\mathbb{R}^N)$ such
that $\|\tilde{T}_j\|_{({L^{q}(\mathbb{R}^N)})^*} = \|h_j\|_{{L^{q'}(\mathbb{R}^N)}}$ and
$$\tilde{T}_j(\phi) = \int_{\mathbb{R}^N}{h_j \phi \dd x}.$$
Hence, for all $\phi \in L^q(\mathbb{R}^N)$ we have
$$\int_{\mathbb{R}^N}^{}{h_j\phi \dd x}-\int_{\mathbb{R}^N}^{}{\frac{f(u_j)\phi}{\|u_j\|_{s,p}^{p-1}} \dd x} =0.$$

Now let us consider the case $sp=N$. It follows from Lemma \ref{limit} that
$$\int_{\mathbb{R}^N}^{}{\left(\frac{|f(u_j)|}{\|u_j\|_{s,p}^{p-1}}\right)^{q'}\dd x}
\leq\frac{C^{q'}}{\|u_j\|_{s,p}^{q'(p-1)}}
\int_{\mathbb{R}^N}^{}\left(Q(|u_j|)\right)^{q'}\dd x.$$
Since $\left(Q(|u_j|)\right)^{q'} \in L^1(\mathbb{R}^N)$, it
follows that
$$\frac{f(u_j)}{\|u_j\|_{s,p}^{p-1}} \in L^{q'}(\mathbb{R}^N).$$

In the case $sp<N$, ($f2$) implies that
$$\int_{\mathbb{R}^N}{\left(\frac{|f(u_j)|}{\|u_j\|_{s,p}^{p-1}}\right)^{q'}\dd x} \leq \frac{C^{q'}}{\|u_j\|_{s,p}^{q'(p-1)}} \int_{\mathbb{R}^N}^{}{\left(|u_j|^{p_{s}^*-1}\right)^{q'} \dd x}.$$
Since $\left(|u_j|^{p_{s}^*-1}\right)^{q'} \in L^1(\mathbb{R}^N)$, we also have
$$\frac{f(u_j)}{\|u_j\|_{s,p}^{p-1}} \in L^{q'}(\mathbb{R}^N).$$

Thus, if $sp\leq N$, we conclude that
$$h_j-\frac{f(u_j)}{\|u_j\|_{s,p}^{p-1}} \in L^{q'}(\mathbb{R}^N) \subset L_{loc}^{1}(\mathbb{R}^N)\quad \text{and}\quad h_j(x)=\frac{f(u_j(x))}{\|u_j\|_{s,p}^{p-1}}\ \  a.e. \textrm{ in } \mathbb{R}^N.$$

Taking $\phi=v_j$ yields
$$\left|  \int_{\mathbb{R}^N}^{}{\frac{f(u_j)}{\|u_j\|_{s,p}^{p-1}}v_j \dd x}\right| \leq \left\|\frac{f(u_j)}{\|u_j\|_{s,p}^{p-1}}\right\|_{q'} \|v_j\|_q = \|h_j\|_{q'}\|v_j\|_q \leq K \|v_j\|_q. $$
But $v_j\to 0$ in $L^q(\mathbb{R}^N)$ implies
$$\int_{\mathbb{R}^N}^{}{\frac{f(u_j)v_j}{\|u_j\|_{s,p}^{p-1}} \dd x} \to 0.$$ Therefore,
$$\|v_j\|_{s,p}^p=\frac{I'(u_j)v_j}{\|u_j\|_{s,p}^{p-1}} + \int_{\mathbb{R}^N}^{}{\frac{f(u_j)}{\|u_j\|_{s,p}^{p-1}}v_j \dd x} \to 0 $$ and we obtain that $v_j \to 0 \textrm{ in }
W_{r}^{s,p}(\mathbb{R}^N)$. This a contradiction, since
$\|v_j\|_{s,p}=1$. We conclude that the $(PS)_c$-sequence must be bounded.
\end{proof}

\begin{lemma}
Passing to a subsequence if necessary, the $(PS)_c$-sequence $(u_j)$
converges strongly in $W_{r}^{s,p}(\mathbb{R}^N)$.
\end{lemma}

\begin{proof} The Lemma \ref{lemLTDA} guarantees that $(u_j)$ is bounded. So, we can suppose that  $u_j\rightharpoonup u \in W_{r}^{s,p}(\mathbb{R}^N)$. It follows
that $u_j \to u$ in $L^q(\mathbb{R}^N)$,  for any $q \in (p,
p_{s}^*)$ and $u_j(x)\to u(x),\ \ a.e. \textrm{ in } \mathbb{R}^N.$
Therefore, $(u_j-u)$is a bounded sequence in
$W_{r}^{s,p}(\mathbb{R}^N)$. Since $I'(u_j) \to 0$ strongly on dual
of the $W_{r}^{s,p}(\mathbb{R}^N)$, we have
\begin{equation*}
I'(u_j)\cdot (u_j-u) \to 0
\end{equation*}
Hence,
\begin{align*}\langle (-\Delta_p)^s u_j, v_j \rangle %&
%=I'(u_j)(u_j-u)-\int_{\mathbb{R}^N}{|u_j|^{p-2}u_j(u_j-u) \dd x} + \int_{\mathbb{R}^N}^{}{f(u_j)(u_j-u) \dd x}\\
& = \int_{\mathbb{R}^N}^{}{f(u_j)(u_j-u) \dd
x}-\int_{\mathbb{R}^N}{|u_j|^{p-2}u_j(u_j-u) \dd x} +
o(1)\end{align*} and we conclude that
$$\langle (-\Delta_p)^s u_j, u_j-u \rangle  \leq \int_{\mathbb{R}^N}^{}{|f(u_j)| |u_j-u| \dd x}+\int_{\mathbb{R}^N}{|u_j|^{p-1}|u_j-u| \dd x} + o(1).$$

Since $\|u_j\|_{p}^p < \infty$, by applying H\"older's inequality we obtain
\begin{align*}\int_{\mathbb{R}^N}{|u_j|^{p-1}|u_j-u|\dd x} \leq \left(\int_{\mathbb{R}^N}{|u_j|^p \dd x}\right)^{\frac{1}{p'}}\left(\int_{\mathbb{R}^N}{|u_j-u|^p \dd x}\right)^{\frac{1}{p}} \leq \tilde{C}\|u_j-u\|_p.
\end{align*}

Let us now suppose that $sp=N$. It follows from Lemma \ref{limit} that
\begin{equation*}
\int_{\mathbb{R}^N}^{}{|f(u_j)| |u_j-u| \dd x} \leq
C\int_{\mathbb{R}^N}^{}Q(|u_j|)|u_j-u| \dd x.
\end{equation*}
If $\alpha \in(0,\alpha_0)$, we can take $r>1$ such that $\alpha
r \in (0, \alpha_0)$. A new application of H\"older's inequality yields
\begin{equation*}
\int_{\mathbb{R}^N}^{}{|f(u_j)| |u_j-u| \dd x}\leq
C\left(\int_{\mathbb{R}^N}^{}(Q(|u_j|))^{r'} \dd
x\right)^{\frac{1}{r'}}\left(\int_{\mathbb{R}^N}^{}{|u_j-u|^r \dd
x}\right)^{\frac{1}{r}}.
\end{equation*}
Therefore, we conclude from the Trudinger-Moser inequality that
$$\int_{\mathbb{R}^N}^{}{|f(u_j)| |u_j-u|
\dd x} \leq C \|u_j-u\|_r$$ and it follows that $\langle A(u_j), u_j-u \rangle
\leq \tilde{C}\|u_j-u\|_p+C\|u_j-u\|_r +o(1) \to 0, \textrm{ as } j
\to \infty$.

If, however, $sp<N$, it follows from $(f2)$ that
\begin{equation*}
\int_{\mathbb{R}^N}^{}{|f(u_j)| |u_j-u| \dd x} \leq
C\int_{\mathbb{R}^N}^{}{|u_j|^{p_{s}^*-1}|u_j-u| \dd x}.
\end{equation*}
Now, taking $r\in [p,p_{s}^*]$ and applying H\"older's inequality we have
\begin{equation*} \int_{\mathbb{R}^N}^{}{|f(u_j)| |u_j-u| \dd x} \leq C\left(\int_{\mathbb{R}^N}^{}{|u_j|^{(p_{s}^*-1)r'}\dd x}\right)^{\frac{1}{r'}}\left(\int_{\mathbb{R}^N}^{}{|u_j-u|^r \dd x}\right)^{\frac{1}{r}}.
\end{equation*}
Since $\left(|u_j|^{(p_{s}^*-1)}\right)^{r'} \in L^1(\mathbb{R}^N),$
it follows that
$$\int_{\mathbb{R}^N}^{}{|f(u_j)| |u_j-u| \dd x} \leq \tilde{C}_1\|u_j-u\|_r.$$
Thus, $\langle A(u_j), u_j-u \rangle \leq
\tilde{C}\|u_j-u\|_p+\tilde{C}_1\|u_j-u\|_r +o(1) \to 0, \textrm{ as
} j \to \infty.$

So, if $sp\leq N$ we have
\begin{align*}\|u_j-u\|_{s,p}^p&=\langle A(u_j), u_j-u \rangle+\int_{\mathbb{R}^N}^{}{|u_j|^{p-2}u_j
(u_j-u) \dd x} - I'(u)\cdot(u_j-u)\\
&\quad -\int_{\mathbb{R}^p}{f(u)(u_j-u) \dd x} +
\int_{\mathbb{R}^N}^{}{f(u_j)(u_j-u) \dd x} \to 0
\end{align*}
and we are done.
\end{proof}

For every $n\in \mathbb{N}$ we set
\begin{equation*}
b_n=\inf_{\gamma\in\Gamma_n}{\max_{\sigma\in
\mathbb{D}_n}{I(\gamma(\sigma))}},
\end{equation*}
where $\Gamma_n=\left\{\gamma \in
\mathcal{C}(\mathbb{D}_n,W_{r}^{s,p}(\mathbb{R}^N); \gamma \
\textrm{ is odd and } \ \gamma=\gamma_n  \ \textrm{ in } \
\partial\mathbb{D}_n \right\}$ and $\mathbb{D}_n $ denotes the unitary disc in $\mathbb{R}^n$, with $\partial\mathbb{D}_n=\mathbb{S}^{n-1}$.
We define
$$\tilde{\gamma}_n(\sigma)=\left\{\begin{array}{cc}
    |\sigma|\gamma_n(\frac{\sigma}{|\sigma|}), & \sigma\in \mathbb{D}_n\setminus \{0\}, \\
    0, & \sigma=0.
  \end{array}
  \right.
$$
Since $\tilde{\gamma}_n\in \Gamma_n$, we have $\Gamma_n \neq
\emptyset$ for all $n\in\mathbb{N}$.

Considering the $(PS)$-sequence $(u_j)$, since $I(u_j)\to  b_n$ and
$$\{u\in W_{r}^{s,p}(\mathbb{R}^N); \|u\|_{s,p}=\rho\}\cap
\gamma(\mathbb{D}_n) \neq \emptyset, \ \forall \gamma \in
\Gamma_n,$$ it follows that $\beta \leq b_n.$ Hence, $0<\beta \leq
b_n, $ for any $n\in\mathbb{N}$.

We now prove that the functional $I(u)$ possesses an unbounded
sequence of critical values.

\begin{lemma}It holds
\begin{enumerate}
\item[$(i)$] $b_n$ is a critical value of $I(u)$ for any $n\in \mathbb{N}$;

\item[$(ii)$] $ b_n \to \infty$ as $n\to \infty.$
\end{enumerate}
\end{lemma}
\begin{proof} We have that $b_n$ is critical value for any $n\in\mathbb{N}$ as a
consequence of the symmetric mountain pass theorem.

As in \cite[Chapter 9]{PHR}, we define
$$\tilde{\Gamma_n}=\left\{h(\overline{\mathbb{D}_m\setminus Z}) ; h\in \Gamma_n, m\geq n, Z\in \mathcal{E}_m \textrm{
and } gen(Z)\leq m-n\right\},$$ where $\mathcal{E}_m$ is a family of
closed subsets $A\subset \mathbb{R}^m\setminus \{0\}$ such that
$A=-A$ and $gen(A)$ is the Krasnoselski genus of $A$.

We define a sequence $(d_n)$ of min-max values for the $I(u)$ as
$$d_n=\inf_{A\in\tilde{\Gamma_n}}{\max_{u\in A}{I(u)}}.$$
It follows that
$$d_n\leq b_n \ \textrm{ and }  \ d_n\leq d_{n+1}, \forall n \in \mathbb{N}.$$

Since $I(u)$ satisfies the $(PS)$-condition, it follows that $$d_n
\to \infty \textrm{ as } n\to \infty.$$ Since  $d_n\leq b_n, \forall
n \in \mathbb{N},$ we have that
$$b_n \to \infty \textrm{
as } n\to \infty.$$
The proof is complete if $0<s<1$.

In the limit case $s = 1$ the adequate Sobolev space is $W^{1,p}(\mathbb{R}^p)$ and results equivalent to those presented in the Introduction are still valid (see, e.g., in \cite{BS, Bre}). Hence, the same arguments exposed above are easily adaptable in order to obtain the result.
\end{proof}

\section{Ground state}
In this section we follow \cite{AFS}. Observe that the case $sp<N$ requires the Poho\v{z}aev identity for the fractional $p$-Laplacian operator. 

Let us consider the set of non trivial solutions of
the problem \eqref{Prob1}, that is, the set
\begin{align*}\Sigma =\left\{u \in W^{s,p}(\mathbb{R}^N)\backslash \{ 0 \} \ : \
    I'(u)=0\right\}.
\end{align*}
We denote by $m$ its least energy level (or ground state level), that is,
\begin{align*}m=\inf_{u \in \Sigma}{I(u)}.
\end{align*}
In the case $sp<N$, we also introduce the set
\begin{equation}\label{M}
    \mathcal{M}=\left\{u \in W^{s,p}(\mathbb{R}^N)\backslash \{ 0 \} \, :
    \, \int_{\mathbb{R}^N}{G(u) \dd x} = 1\right\},
\end{equation}
where $G(t)=F(t)-\displaystyle\frac{|t|^p}{p}$ is the primitive of
$g(t)$. 

In the case $sp=N$, we define $\mathcal{M}=\left\{u \in W^{s,p}(\mathbb{R}^N)\backslash \{ 0 \}\,:\,\int_{\mathbb{R}^N}{G(u) \dd x}=0\right\}$.

Furthermore, we will need to consider the space
$$\mathcal{D}^{s,p}(\mathbb{R}^N):=\{  u\in L^{p_{s}^*}(\mathbb{R}^N)\,:\, [u]_{s,p}^{p}< \infty \}.$$
It is well known that the following inequality holds
$$S\left[\int_{\mathbb{R}^N}{|u|^{p_{s}^*}}\right]^{\frac{p}{p_{s}^*}} \leq [u]_{s,p}^p, \quad \forall u \in  \mathcal{D}^{s,p}(\mathbb{R}^N),$$
where $S>0$ is the best constant of Sobolev embedding, see
\cite{VM}.

\textbf{The case $sp<N$.}

The simple proof of the next result can be found in \cite[Lemma 2.2]{AFS}.
\begin{lemma} The set  $\mathcal{M}$ defined in \eqref{M} is not empty and a $\mathcal{C}^1$ manifold.
\end{lemma}

We denote
\begin{equation*}
    T(u)= \displaystyle\frac{1}{p}[u]_{s,p}^p \qquad D = \inf_{u \in
        \mathcal{M}}{T(u)}.
\end{equation*}
Observe that
\begin{align*}
    pD = \inf_{u \in \mathcal{M}}{[u]_{s,p}^p}.
\end{align*}

It is worth to point out that if we define the $C^1$ functional
\[J(u)=\int_{\mathbb{R}^N}G(u)\dd x- 1,\]
it follows from ($f5$) that $u\in\mathcal{M}$ implies
\begin{equation*}
    J'(u)\cdot u=\int_{\mathbb{R}^N}(f(u)u-u^p)\dd x=\int_{\mathbb{R}^N}(f(u)u-pF(u))\dd x+p\int_{\mathbb{R}^N}G(u)\dd x\geq p.
\end{equation*}

We also define the mini-max level associated to the functional $I$
\begin{equation}\label{b}
    b=\inf_{\gamma \in \Gamma}{\max_{t\in[0,1]}{I(\gamma(t))}},
\end{equation}
where
\begin{align*}\Gamma=\{\gamma\in\mathcal{C}([0,1],W^{s,p}(\mathbb{R}^N)):
    \gamma(0)=0, I(\gamma(1))<0\}.
\end{align*}
Since $I$ satisfies the Mountain Pass Geometry in
$W^{s,p}_r(\mathbb{R}^N)$, (see Lemma \ref{lemgeo}), it follows that
$\Gamma \neq \emptyset$.

Also in the case $sp<N$, we make use of a result which is not fully
proved (see \cite{LMS2, Brasco}), namely, that a weak solution $u\in
\mathcal{D}^{s,p}(\mathbb{R}^N)\cap L^\infty(\mathbb{R}^N)$ of
problem \eqref{Prob1} satisfies the Poho\v{z}aev identity
\[\frac{N-sp}{p}[u]^p_{s,p}=N\int_{\mathbb{R}^N}G(u)\dd x.\]
Therefore, we define the set
\begin{align*}\mathcal{P}=\left\{u \in W^{s,p}(\mathbb{R}^N)\setminus \{ 0 \} \,:\, \displaystyle\frac{N-sp}{p}[u]_{s,p}^p=N \int_{\mathbb{R}^N}^{}{G(u)\dd x}\right\}.
\end{align*}
We denote by
\begin{align*}\tilde{p}=\inf_{u \in \mathcal{P}}{I(u)}
\end{align*}
and adapting the arguments in \cite[Lemma 2.4]{LMS} we have
$$ \tilde{p}= \displaystyle\frac{s}{N}\left(\frac{N-sp}{pN}\right)^{(N-sp)/sp}(pD)^{N/sp},$$
(see also \cite[Theorem 6.3]{Brasco}).

With minor changes, the proof of the next result is given in
\cite[Lemma 2.1 ]{AFS}.
\begin{lemma} It holds $\tilde{p} \leq b$.
\end{lemma}
The two next results are Lemma 2.3 and Lemma 2.4 in \cite{AFS}, respectively.
\begin{lemma} Any minimizing sequence $(u_j)\subset \mathcal{M}$ for $T$ is bounded in $W^{s,p}(\mathbb{R}^N).$
\end{lemma}
\begin{lemma}
    The number $D$ is positive.
\end{lemma}

From Ekeland's Variational Principle there are $(u_n)\subset
\mathcal{M}$ and a sequence of Lagrange multipliers
$(\lambda_n)\subset \mathbb{R}$ such that $T(u_n)\to D$ and
$T'(u_n)-\lambda_n J'(u_n)\to 0$ in the dual space
$(W^{s,p}(\mathbb{R}^N))'$. Furthermore, we have already proved that
any minimizing sequence is bounded in $W^{s,p}(\mathbb{R}^N)$.

\textit{Mutatis mutandis}, the proof of the next result is given in \cite[Lemma 2.5 ]{AFS}.
\begin{lemma}\label{lmult} The sequence $(\lambda_n)$ of Lagrange multipliers is bounded. More precisely, it holds \[0\leq \liminf_{n\to\infty}\lambda_n\leq \limsup_{n\to\infty}\lambda_n\leq D.\]
\end{lemma}

As a consequence of Lemma \ref{lmult}, passing to a subsequence if necessary, we can suppose that $\lambda_n\to \lambda^*\in (0,D]$.

The proof of the next two results is obtained by adapting those given in
Lemma 2.6 and Lemma 2.7 in \cite{AFS}, respectively.
\begin{lemma} Any minimizing sequence $(u_n)$ for $T$ can be assumed
radially symmetric around the origin and non-negative.
\end{lemma}

\begin{lemma} Suppose that $v_n:=u_n-u \rightharpoonup 0$ in $W^{s,p}(\mathbb{R}^N)$ and $[v_n]_{s,p}^p\to L>0.$
Then $$D\geq p^{-sp/N}S$$
\end{lemma}

\begin{lemma}It holds
$$b<\displaystyle\frac{s}{N}\left(\frac{N-sp}{pN}\right)^{(N-sp)/sp}p^{N-sp/sp}S^{N/sp}$$
\end{lemma}
This is Lemma 2.8 in \cite{AFS}.

\begin{lemma}If $u_n \rightharpoonup u$ in $W^{s,p}(\mathbb{R}^N),$
then $u_n\to u$ in $\mathcal{D}^{s,p}(\mathbb{R}^N).$ In particular,
$u_n\to u$ in $L^{p^*_s}(\mathbb{R}^N).$
\end{lemma}

This is Lemma 2.9 in \cite{AFS}.

Finally, the proof of Theorem \ref{Thgs} in the case $sp<N$ results
from the proof of \cite[Theorem 1.2]{AFS}.

\textbf{The case $sp=N$.}

The proof of the next two results follows immediately from the corresponding results in \cite{AFS}, namely Lemmas 3.1 and 3.2.
\begin{lemma} The set  $\mathcal{M}$ defined in \eqref{M} is not empty and a $\mathcal{C}^1$ manifold.
\end{lemma}
\begin{lemma}Assume that $f$ satisfies $(f1)$ and $(f5)$ and let $(u_n)$ be a sequence in $W^{s,p}_r(\mathbb{R}^N)$ such that  $u_n \rightharpoonup u$ in $W^{s,p}(\mathbb{R}^N)$ such that
$$\sup_{n}{[u_n]_{s,p}^p}=\rho<1\quad \textrm{and}\quad \sup_{n}{\|u_n\|_{p}^p}=M < \infty.$$
Then, $$\int_{\mathbb{R}^N}{F(u_n)} \to \int_{\mathbb{R}^N}{F(u)}.$$
\end{lemma}

The relation between the ground state level and the mini-max level
defined in \eqref{b} is given by the following result, which is
proved in \cite[Lemma 3.3, Lemma 3.4]{AFS}.
\begin{lemma} It holds $0<D\leq b$.
\end{lemma}

The next result follows easily by adapting \cite[Lemma 3.5]{AFS}.
\begin{lemma} If $\mu > \mu^*$ then $b< \frac{1}{p}$
    \end{lemma}

Finally, the proof of Theorem \ref{Thgs} in the case $sp=N$ results from the proof of \cite[Theorem 1.3]{AFS}.

\end{document}